\documentclass[a4paper,11pt]{amsart}

\usepackage[utf8]{inputenc}
\usepackage{lmodern}
\usepackage{url}
\usepackage{amsmath, amsthm, amssymb, amscd, accents, bm}
\usepackage{mathtools}
\usepackage{mdwlist}
\usepackage{paralist}
\usepackage{graphicx}
\usepackage{epstopdf}
\usepackage{datetime}
\usepackage{afterpage}
\usepackage[dvipsnames]{xcolor}
\usepackage{hyperref}
\usepackage[ruled,vlined]{algorithm2e}
\usepackage{caption}
\usepackage{subcaption}
\usepackage{marginnote}
\usepackage{soul}

\usepackage[sort,nocompress]{cite}
\mathtoolsset{showonlyrefs}

\usepackage[foot]{amsaddr}

\usepackage{mathtools}

\newcommand\rurl[1]{%
  \href{https://#1}{\nolinkurl{#1}}%
}

\newtheorem{definition}{Definition}[section]
\newtheorem{theorem}[definition]{Theorem}
\newtheorem{lemma}[definition]{Lemma}
\newtheorem{example}[definition]{Example}
\newtheorem{corollary}[definition]{Corollary}
\newtheorem{proposition}[definition]{Proposition}

\newtheorem*{theorem*}{Theorem}

\newtheoremstyle{rem}
  {}
  {}
  {}            
  {}            
  {\bfseries}   
  {.}            
  { }    
  {}            
\theoremstyle{rem}
\newtheorem{remark}[definition]{Remark}

\def\N{{\mathbb N}}
\def\Z{{\mathbb Z}}
\def\R{{\mathbb R}}
\def\T{{\mathbb T}}
\def\C{{\mathbb C}}

\newcommand{\Rd}{{\R^d}}
\newcommand{\Cd}{{\C^d}}

\newcommand{\lt}{{L^2(\R)}}
\newcommand{\ltd}{{L^2(\R^d)}}

\newcommand{\ift}{{\mathcal{F}^{-1}}}
\newcommand{\ft}{{\mathcal{F}}}

\makeatletter
\newcommand\thankssymb[1]{\textsuperscript{\@fnsymbol{#1}}}

\makeatletter
\def\@makefnmark{%
  \leavevmode
  \raise.9ex\hbox{\fontsize\sf@size\z@\normalfont\tiny\@thefnmark}}

\makeatletter
\def\bign#1{\mathclose{\hbox{$\left#1\vbox to8.5\p@{}\right.\n@space$}}\mathopen{}}

\usepackage{eqparbox}
\usepackage{makebox}
\newcommand\textoverset[3][]{\mathrel{\overset{\scriptsize{\eqmakebox[#1]{#2}}}{#3}}}

\begin{document}

\title[Phaseless sampling on square-root lattices]{Phaseless sampling on square-root lattices}
\author[Philipp Grohs]{Philipp Grohs\thankssymb{1}\textsuperscript{,}\thankssymb{2}\textsuperscript{,}\thankssymb{3}}
\address{\thankssymb{1}Faculty of Mathematics, University of Vienna, Oskar-Morgenstern-Platz 1, 1090 Vienna, Austria}
\address{\thankssymb{2}Research Network DataScience@UniVie, University of Vienna, Kolingasse 14-16, 1090 Vienna, Austria}
\address{\thankssymb{3}Johann Radon Institute of Applied and Computational Mathematics, Austrian Academy of Sciences, Altenbergstrasse 69, 4040 Linz, Austria}
\email{philipp.grohs@univie.ac.at}
\author[Lukas Liehr]{Lukas Liehr\thankssymb{1}}
\email{lukas.liehr@univie.ac.at}

\date{\today}
\maketitle

\begin{abstract}
Due to its appearance in a remarkably wide field of applications, such as audio processing and coherent diffraction imaging, the short-time Fourier transform (STFT) phase retrieval problem has seen a great deal of attention in recent years. A central problem in STFT phase retrieval concerns the question for which window functions $g \in \ltd$ and which sampling sets $\Lambda \subseteq \R^{2d}$ is every $f \in \ltd$ uniquely determined (up to a global phase factor) by phaseless samples of the form
$$
|V_gf(\Lambda)| = \left \{ |V_gf(\lambda)| : \lambda \in \Lambda \right \},
$$ 
where $V_gf$ denotes the short-time Fourier transform (STFT) of $f$ with respect to $g$. The investigation of this question constitutes a key step towards making the problem computationally tractable. However, it deviates from ordinary sampling tasks in a fundamental and subtle manner: recent results demonstrate that uniqueness is unachievable if $\Lambda$ is a lattice, i.e $\Lambda = A\Z^{2d}, A \in \mathrm{GL}(2d,\R)$. Driven by this discretization barrier, the present article centers around the initiation of a novel sampling scheme which allows for unique recovery of any square-integrable function via phaseless STFT-sampling. Specifically, we show that square-root lattices, i.e., sets of the form
$$
\Lambda = A \left ( \sqrt{\Z} \right )^{2d}, \ \sqrt{\Z} = \{ \pm \sqrt{n} : n \in \N_0 \},
$$
guarantee uniqueness of the STFT phase retrieval problem. The result holds for a large class of window functions, including Gaussians.
\end{abstract}

\section{Introduction}

Consider a window function $g \in \ltd$ and a set $\Lambda \subseteq \R^{2d}$ of sampling locations. Let $V_gf : \R^{2d} \to \C$ denote the short-time Fourier transform (STFT) of $f \in \ltd$ with respect to $g$, given by
$$
V_gf(x,\omega) = \int_{\Rd} f(t)\overline{g(t-x)}e^{-2\pi i \omega \cdot t} \, dt.
$$
The well-established theory in time-frequency analysis ensures that every $f \in \ltd$ is uniquely determined by the STFT-samples $\{ V_gf(\lambda) : \lambda \in \Lambda \}$, provided that the window function $g$ satisfies certain mild conditions, and $\Lambda = A\Z^{2d}, A \in \mathrm{GL}(2d,\R),$ is a sufficiently dense lattice. Consequently, square-integrable functions can be distinguished via sampling of their STFTs. Suppose now that merely the absolute values of the STFT-samples are given, i.e., the sampling set is of the form
\begin{equation}\label{eq:phaseless_samples}
    |V_gf(\Lambda)| \coloneqq  \left \{ |V_gf(\lambda)| : \lambda \in \Lambda \right \}.
\end{equation}
This leads to the phaseless sampling problem and the natural question whether a distinction of square-integrable functions is possible by considering the respective phaseless STFT samples as given in \eqref{eq:phaseless_samples}. The investigation of this problem is of utmost importance in a series of applications, prominent examples include coherent diffraction imaging \cite{appl5,appl7}, audio processing \cite{appl8}, and quantum mechanics \cite{appl1}. It is known as the \emph{uniqueness problem in STFT phase retrieval} and has received considerable attention in recent years. Since $|V_gf(\Lambda)|=|V_g(\tau f)(\Lambda)|$ whenever $\tau \in \T \coloneqq \{ z \in \C : |z|=1 \}$ is a complex number of modulus one, it is evident that a distinction from phaseless samples of the form \eqref{eq:phaseless_samples} is only achievable up to the ambiguity arising from multiplication by a number in $\T$.
To that end, we define for two functions $f,h \in \ltd$ the equivalence relation
$$
f \sim h \iff \exists \tau \in \T : f = \tau h,
$$
and say that $f$ and $h$ agree up to a global phase if $f \sim h$. Compared to conventional sampling of the STFT, the uniqueness problem posed by phaseless samples takes a fundamentally different direction. Recently established discretization barriers have revealed that if $\Lambda$ is a lattice, then functions within $\ltd$ cannot be uniquely determined by samples of the form of $|V_gf(\Lambda)|$, regardless of the lattice choice $\Lambda$ or the window function $g$ \cite{grohsLiehrJFAA,grohsLiehr4,alaifari2020phase}. Consequently, addressing the uniqueness problem in the absence of phase information calls for new techniques and alternative sampling schemes.

The objective of the present article centers around the derivation of phaseless sampling results from square-root lattices. Precisely, we determine a function space $\mathcal{O}_a^b(\Cd)$ (depending on a decay parameter $a$ and a growth parameter $b$) which consists of analytic functions on $\C^d$ such that the restriction to $\Rd$ of every $g \in \mathcal{O}_a^b(\Cd)$ is an element of $\ltd$. This function space has the following property: whenever $g$ belongs to $\mathcal{O}_a^b(\Cd)$, then every $f \in \ltd$ is determined up to a global phase by $|V_gf(\Lambda)|$, provided that $\Lambda$ is a square-root lattice, 
$$
\Lambda = A\left (\sqrt{\Z} \right )^{2d}, \quad \sqrt{\Z} \coloneqq \{ \pm \sqrt{n} : n \in \N_0 \}, \quad A \in \mathrm{GL}(2d,\R),
$$
which satisfies a generalized density condition (see Section \ref{sec:main_results} for further clarifications). Thus, sampling on ordinary lattices never yields uniqueness, while a replacement by a square-root lattice does the job.
It is worth mentioning that $\mathcal{O}_a^b(\Cd)$ contains Gaussians, the most important window functions in contemporary STFT phase retrieval research. Moreover, a suitable choice of $a$ and $b$ implies that $\mathcal{O}_a^b(\Cd)$ is dense in $\ltd$.

\subsection{Main results}\label{sec:main_results}

Recall that a lattice in $\Rd$ is a set of the form $A\Z^d$ for some invertible matrix $A \in \mathrm{GL}(d,\R)$. We call a subset $\Lambda \subseteq \Rd$ a \emph{square-root lattice} if there exists an invertible matrix $A \in \mathrm{GL}(d,\R)$ such that
$$
\Lambda = A\left (\sqrt{\Z} \right)^d = \left \{ Az : z \in \underbrace{\sqrt{\Z} \times \dots \times \sqrt{\Z}}_{d \ \mathrm{times}} \right \},
$$
where $\sqrt{\Z} \coloneqq \{ \pm \sqrt{n} : n \in \N_0 \}$. An example of a square-root lattice is depicted in Figure \eqref{fig:comparison}. In analogy to an ordinary lattice, the matrix $A$ is called the generating matrix of the square-root lattice. A square-root lattice $\Lambda$ is called rectangular if its generating matrix is diagonal. The present article establishes uniqueness results for the STFT phase retrieval problem via sampling on square-root lattices. The window functions which define the STFT and for which square-root sampling implies uniqueness belong to a space of analytic functions. Before defining this space, we settle some notation. Whenever $z \in \Cd$, we write $z=x+iy$, where $x \in \Rd$ is the real part of $z$, and $y \in \Rd$ is the imaginary part of $z$. For two functions $v,w : \Cd \to [0,\infty)$, we write $v \lesssim w$, if and only if there exists a positive constant $c>0$ such that $v(z) \leq c w(z)$ for every $z \in \Cd$.

The space of analytic functions we are dealing with will be denoted by $\mathcal{O}_a^b(\Cd)$. For $a,b \in \R_{>0}^d \coloneqq (0,\infty)^d$ it is defined by
$$
    \mathcal{O}_{a}^b(\C^d) \coloneqq \left \{  F \in \mathcal{O}(\C^d) : |F(x+iy)| \lesssim \prod_{j=1}^d e^{-a_j x_j^2}e^{b_jy_j^2} \right \},
$$
where $\mathcal{O}(\C^d)$ denotes the space of entire functions of $d$ complex variables. The function class $\mathcal{O}_{a}^b(\C^d)$ is a linear subspace of $\mathcal{O}(\Cd) \cap \ltd$ (the space of all entire functions which are square-integrable on $\Rd$), and if $a<b$ (i.e., $a_j<b_j$ for every $j \in \{ 1, \dots, d \}$), then $\mathcal{O}_{a}^b(\C^d)$ is dense in $\ltd$. Note that we will repeatedly identify entire functions with their restriction on $\Rd$, and vice versa, functions on $\Rd$ with their analytic extension on $\C^d$ (provided the extension exists). We refer to Section \ref{sec:function_space_properties} for a detailed discussion of properties of $\mathcal{O}_{a}^b(\C^d)$. Window functions which belong to the function space $\mathcal{O}_{a}^b(\C^d)$ lead to phaseless sampling results from square-root lattices. This is the content of the following theorem which constitutes the main result of the article.

\begin{theorem}\label{thm:main_result}
Let $a,b \in \R^d_{>0}$, and let $0 \neq \varphi \in \mathcal{O}^a_b(\C^d)$ be a window function. Suppose that $\Lambda=A(\sqrt{\Z})^{2d}$ is a rectangular square-root lattice such that the generating matrix $A = \mathrm{diag}(\tau_1,\dots,\tau_d,\nu_1,\dots,\nu_d)$ with $\tau,\nu \in \R^d_{>0}$ satisfies
$$
\tau_j < \sqrt{\frac{1}{2b_j e}}, \quad \nu_j < \sqrt{\frac{a_j}{2\pi^2e}}, \quad j \in \{ 1,\dots,d \}.
$$
Then the following statements are equivalent for every $f,h \in \ltd$:
\begin{enumerate}
    \item $|V_\varphi f(\lambda)| = |V_\varphi h(\lambda)|$ for every $\lambda \in \Lambda$,
    \item $f \sim h$.
\end{enumerate}
\end{theorem}

Notice that for $\gamma > 0$ and for a suitable choice of $a$ and $b$, the Gaussian $\varphi(x)=e^{-\gamma \| x \|_2^2}, \, \| x \|_2^2 = \sum_{j=1}^d x_j^2$, is an element of $\mathcal{O}^a_b(\C^d)$. In view of both theory and applications, Gaussians represent the most significant window functions as they allow for a connection to the theory of Fock spaces. It is this particular relation which leads to the most extensive results in STFT phase retrieval with Gaussian windows \cite{grohsliehr1,grohsliehr2,ALAIFARI2021401,GrohsRathmair,GrohsRathmair2,daub,multiwindow}. We therefore state two separate results on square-root sampling for Gaussian-type windows, where we also allow for multiplication by entire functions of exponential type. The first consequence of Theorem \ref{thm:main_result} reads as follows.

\begin{corollary}\label{cor:polynomial_times_gauss}
    Let $\gamma > 0$, let $p : \C^d \to \C$ be an entire function of exponential type, and let $\varphi \in \ltd$ be the window function
    $
    \varphi(x) = p(x)e^{-\gamma \| x \|_2^2}.
    $
    If $\alpha,\beta > 0$ satisfy
    $$
    \alpha < \sqrt{\frac{1}{2\gamma e}}, \quad \beta < \sqrt{\frac{\gamma}{2 \pi^2 e}},
    $$
    then the following statements are equivalent for every $f,h \in \ltd$:
    \begin{enumerate}
    \item $|V_\varphi f(\lambda)| = |V_\varphi h(\lambda)|$ for every $\lambda \in \alpha (\sqrt{\Z})^d \times \beta (\sqrt{\Z})^d$,
    \item $f \sim h$.
\end{enumerate}
\end{corollary}

Every polynomial is an entire function of exponential type.  Hermite functions are products of Gaussians with polynomials. Recall that the Hermite functions on the real line are given by
$$
h_n(t) = \frac{2^{1/4}}{\sqrt{n!}} \left( -\frac{1}{\sqrt{2\pi}} \right)^n e^{\pi t^2} \left( \frac{d}{dt} \right)^n e^{-2\pi t^2}, \quad n \in \N_0.
$$
The system of Hermite basis functions $\{ \mathfrak{h}_k : k \in \N_0^d \} \subseteq \ltd$ is defined via tensorisation,
$$
\mathfrak{h}_k(t_1, \dots, t_d) \coloneqq \prod_{j=1}^d h_{k_j}(t_j).
$$
The prototypical example of a Hermite function is the standard Gaussian $\mathfrak{h}_0$ which is given by $\mathfrak{h}_0(t) = 2^{d/4} e^{-\pi \| t \|_2^2}$. Any other Hermite function is the product of $\mathfrak{h}_0$ with a polynomial. The choice of Hermite functions as window functions implies the following concise statement.

\begin{corollary}\label{cor:hermite}
    Let $\mathfrak{h} \in \ltd$ be an arbitrary Hermite function. If $$0 < \alpha < \sqrt{\frac{1}{2\pi e}},$$ then the following statements are equivalent for every $f,h \in \ltd$:
    \begin{enumerate}
    \item $|V_\mathfrak{h} f(\lambda)| = |V_\mathfrak{h} h(\lambda)|$ for every $\lambda \in \alpha (\sqrt{\Z})^{2d}$,
    \item $f \sim h$.
\end{enumerate}
\end{corollary}

Finally, we note that more general versions of the previous theorems are derived where the generating matrix of the square-root lattice does not need to be diagonal. We refer to Section \ref{sec4} regarding this matter.

\begin{figure}
 \centering
 \hspace*{-1cm}
   \includegraphics[width=14cm,trim={0 1.5cm 0 0},clip]{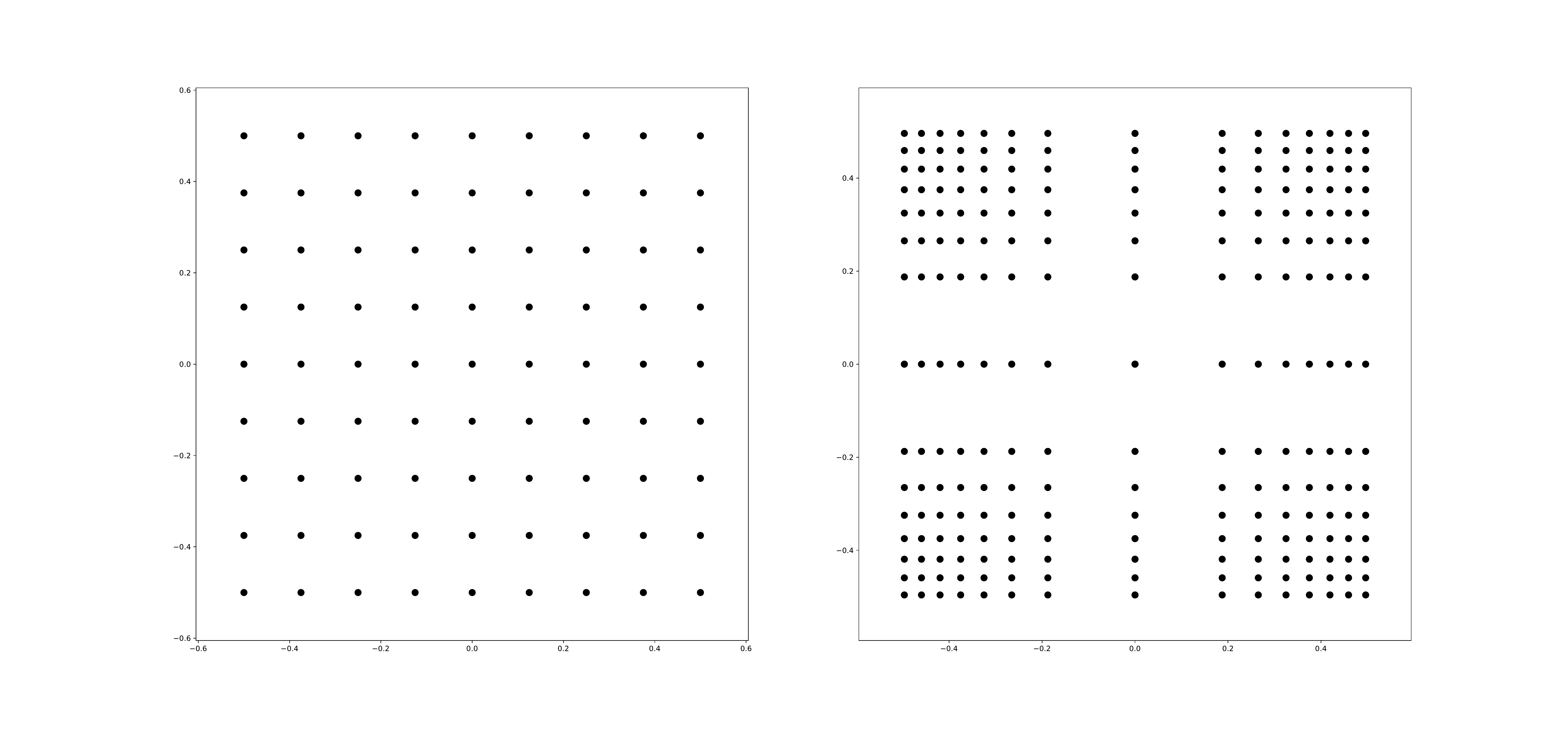}
 \caption{The left figure depicts an ordinary lattice of the form $A\Z^2$ while the right figure visualizes a square-root lattice $B(\sqrt{\Z})^2$ $(A,B \in \mathrm{GL}(2,\R))$.}
 \label{fig:comparison}
 \end{figure}

\subsection{Existing work on square-root sampling}

The purpose of this section is to mention several existing results on square-root sampling. The first one is a sampling result for the short-time Fourier transform with Gaussian window function.

\textbf{Ascensi, Lyubarskii and Seip \cite{ASCENSI2009277}.} Consider the set $\Lambda \subseteq \R^2 \simeq \C$ defined by
\begin{equation}\label{eq:seip}
    \Lambda = \sqrt{2} \left ( \{ 0 \} \times \sqrt{\Z} \right ) \cup \sqrt{2} \left ( \sqrt{\Z} \times \{ 0 \} \right ) \cup \{ (1,0),(0,1)\}.
\end{equation}
According to \cite[Theorem 1]{ASCENSI2009277}, the set $\Lambda$ is a uniqueness set for the Bargmann-Fock space $\mathcal{F}^2(\C)$ which is defined as the collection of all entire functions $F \in \mathcal{O}(\C)$ so that
$$
\int_\C |F(z)|^2 e^{-\pi |z|^2} \, dA(z) < \infty,
$$
where $dA(z)$ denotes the Euclidean area measure on $\C$. The Bargmann-transform constitutes a unitary isomorphism between $\mathcal{F}^2(\C)$ and $\lt$. In addition, it yields an intimate relation between the short-time Fourier transform $V_\varphi$ with window function $\varphi(t)=2^{1/4}e^{-\pi t^2}$ and $\mathcal{F}^2(\C)$ \cite[Section 3.4]{Groechenig}. In view of this one-to-one correspondence, the theorem of Ascensi, Lyubarskii and Seip implies that if $\Lambda$ is defined as in \eqref{eq:seip}, then the following statements are equivalent for every $f,h \in \ltd$:
    \begin{enumerate}
    \item $V_\varphi f(\lambda) = V_\varphi h(\lambda)$ for every $\lambda \in \Lambda$,
    \item $f = h$.
    \end{enumerate}
This shows that in the presence of phase information, square-root samples which are distributed over the time- and frequency axis give rise to unique determination of every $f \in \lt$.

\textbf{Viazovska and Radchenko \cite{viazovska1}.}
Let $\mathcal{S}_e(\R)$ denote the space of even Schwartz functions $f : \R \to \R$. According to \cite[Corollary 1]{viazovska1}, the following two statements are equivalent for every $f,h \in \mathcal{S}_e(\R)$:
\begin{enumerate}
    \item $f(\lambda)=h(\lambda)$ and $\hat{f}(\lambda) = \hat{h}(\lambda)$ for every $\lambda \in \sqrt{\Z_{\geq 0}}$,
    \item $f=h$.
\end{enumerate}
Thus, every even Schwartz function is determined by square-root samples of itself and its Fourier transform. The proof of this statement is based on the theory of modular forms.

\section{Preliminaries}

In this section, we collect several preliminary results which are needed in the remainder of the article, in particular, in the proofs of the assertions presented in Section \ref{sec:main_results}. The focus of study is the function space $\mathcal{O}_{a}^b(\C^d)$ and uniqueness sets for spaces of analytic functions.

\subsection{The function space $\mathcal{O}_{a}^b(\C^d)$}\label{sec:function_space_properties}

Recall that if $a,b \in \R^d_{>0}$, then $\mathcal{O}_a^b(\Cd)$ denotes the collection of all entire functions $F \in \mathcal{O}(\Cd)$ which satisfy the estimate
$$
|F(x+iy)| \lesssim \prod_{j=1}^d e^{-a_j x_j^2}e^{b_jy_j^2}, \quad x,y \in \Rd.
$$
Further, we define
$$
\mathcal{O}^b(\C^d) \coloneqq \left \{ F \in \mathcal{O}(\C^d) : |F(z_1,\dots, z_d)| \lesssim \prod_{j=1}^d e^{b_j|z_j|^2} \right \}.
$$
It is evident that for every $a \in \R^d_{>0}$ we have $\mathcal{O}_{a}^b(\C^d) \subseteq \mathcal{O}^b(\C^d)$. The parameters $a$ and $b$ govern the growth and the decay of $F \in  \mathcal{O}_{a}^b(\C^d)$: the constant $b$ upper bounds the growth of $F$ on $\Cd$ whilst $a$ simultaneously lower bounds the decay rate of $F$ on $\Rd$. It follows, that every $F \in \mathcal{O}_{a}^b(\C^d)$ is an entire function of order at most two in each complex variable. Recall that an entire function $F \in \mathcal{O}(\C)$ is said to be of finite order if there exists a positive number $c>0$ such that
$
|F(z)| \lesssim e^{|z|^c}.
$
The smallest non-negative number $\rho \geq 0$ such that
$$
|F(z)| \lesssim e^{|z|^{\rho + \varepsilon}}
$$
for every $\varepsilon>0$ is called the order of $F$. The order of a non-constant entire function $F$ is given by the formula
$$
\rho = \limsup_{r \to \infty} \frac{\log \log M(r)}{\log r},
$$
where $M(r) = \max_{z \in \partial B_r(0)} |F(z)|$ denotes the maximum-modulus function of $F$, and $\partial B_r(0) \coloneqq \{ z \in \C : |z|=r \}$.
An entire function $G : \Cd \to \C$ is said to be of exponential type if it is of order one in each variable, that is, there exists a positive real number $\sigma > 0$ such that
$$
|G(z)| \lesssim e^{\sigma \| z \|_1}
$$
where $\| z \|_1 = \sum_{j=1}^d |z_j|, \, z = (z_1,\dots,z_d)^T \in \Cd$. The following short lemma makes an assertion about the product of an entire function of exponential type with a function in $\mathcal{O}_a^b(\Cd)$. It will be of frequent use throughout the remainder of the article.

\begin{lemma}\label{lemma:expType}
    Let $a,b \in \R^d_{>0}$, let $F \in \mathcal{O}_a^b(\Cd)$, and let $G$ be an entire function of exponential type. Then for every $\varepsilon = (\varepsilon_1, \dots, \varepsilon_d)^T \in \R_{>0}^d$ with $\varepsilon_j \in (0,a_j)$ for all $j \in \{ 1,\dots, d \}$, it holds that $GF \in \mathcal{O}_{a-\varepsilon}^{b+\varepsilon}(\Cd)$.
\end{lemma}
\begin{proof}
    Since $G$ is of exponential type, there exists a positive real number $\sigma >0$ such that
    $$
    |G(z)| \lesssim  e^{\sigma \| z \|_1} \leq \prod_{j=1}^d e^{\sigma |x_j|} e^{\sigma |y_j|}.
    $$
    Notice that for every $\delta>0$ there exists a constant $C=C(\delta)>0$ such that for every $t \in \R$ it holds that $e^{\sigma |t|} \leq Ce^{\delta t^2}$. Thus,
    $$
    |G(x+iy)F(x+iy)| \lesssim \prod_{j=1}^d e^{-(a_j-\varepsilon_j)x_j^2}e^{(b_j+\varepsilon_j)y_j^2},
    $$
    and therefore $GF \in \mathcal{O}_{a-\varepsilon}^{b+\varepsilon}(\Cd).$
\end{proof}

We proceed with the discussion of several important properties of the function class $\mathcal{O}_{a}^b(\C^d)$.

\begin{proposition}\label{prop:O_properies}
Let $a,b \in \R^d_{>0}$. Then the following properties hold:
\begin{enumerate}
    \item The set $\mathcal{O}_a^b(\Cd)$ is a linear subspace of $\mathcal{O}(\Cd) \cap L^2(\Rd)$.
    \item For every $\mathfrak{a},\mathfrak{b} \in \R_{>0}^d$ such that $\mathfrak{a}_j \in (0,a_j)$ and $\mathfrak{b}_j > b_j$ for all $j \in \{ 1, \dots, d \}$, it holds that
    $$
    \mathcal{O}_a^b(\Cd) \subseteq \mathcal{O}_{\mathfrak{a}}^{\mathfrak{b}}(\Cd).
    $$
    \item If $a<b$, then $\mathcal{O}_a^b(\Cd)$ is infinite-dimensional and dense in $\ltd$.
    \item If $a=b$, then $\mathcal{O}_a^b(\Cd) = \C \varphi$ where $\varphi \in \mathcal{O}(\Cd)$ is the Gaussian $$\varphi(z_1,\dots,z_d) = \prod_{j=1}^d e^{-a_j z_j^2}.$$
    \item If $a_j>b_j$ for some $j \in \{ 1,\dots, d \}$, then $\mathcal{O}_a^b(\Cd) = \{ 0 \}$.
\end{enumerate}
\end{proposition}
\begin{proof}
    The property of $\mathcal{O}_a^b(\Cd)$ being a linear space of entire functions follows directly from its definition. Since every $F \in \mathcal{O}_a^b(\Cd)$ has Gaussian decay on $\Rd$, we further deduce the inclusion $\mathcal{O}_a^b(\Cd) \subseteq L^2(\Rd)$. The inclusion $\mathcal{O}_a^b(\Cd) \subseteq \mathcal{O}_{\mathfrak{a}}^{\mathfrak{b}}(\Cd)$ with $\mathfrak{a},\mathfrak{b}$ defined as above follows at once from the definition of $\mathcal{O}_a^b(\Cd)$. 
    This shows the first two claims of the Proposition. 
    To prove the third statement, suppose that $a<b$ and define $c \coloneqq \frac{1}{2}(a+b)$. Then for every $j \in \{ 1,\dots, d \}$ and every $z \in \Cd$ we have
    \begin{equation}\label{eq:cj}
        c_j > a_j, \quad c_j < b_j, \quad \left |e^{-c_j z_j^2} \right | = e^{-c_j x_j^2} e^{c_j y_j^2}.
    \end{equation}
    This shows that the Gaussian $G \in \mathcal{O}(\Cd)$ defined via
    $$
    G(z) = \prod_{j=1}^d e^{-c_j z_j^2}
    $$
    is an element of $\mathcal{O}_a^b(\Cd)$. Now let $p \in \mathcal{O}(\Cd)$ be a polynomial. By virtue of Lemma \ref{lemma:expType} and the fact that $p$ is an entire function of exponential type, it holds that $pG \in \mathcal{O}_{c-\varepsilon}^{c+\varepsilon}(\Cd)$ for every $\varepsilon = (\varepsilon_1, \dots, \varepsilon_d)^T \in \R_{>0}^d$ with $\varepsilon_j \in (0,a_j)$ for all $j \in \{ 1,\dots, d \}$. In view of Equation \eqref{eq:cj}, it follows that a sufficiently small choice of $\varepsilon_j$ implies that we have both $c_j-\varepsilon_j > a_j$ and $c_j + \varepsilon_j < b_j$ for every $j \in \{1,\dots,d \}$. Thus, the previous inclusion shows that $pG \in \mathcal{O}_a^b(\Cd)$. Since $p$ was arbitrary, every function which is a product of a polynomial with the Gaussian $G$ is an element of $\mathcal{O}_a^b(\Cd)$. In particular, $\mathcal{O}_a^b(\Cd)$ contains a system of scaled Hermite functions, the span of which is known to be dense in $\ltd$. The statement follows from that.
    For the fourth assertion, we observe that if $F \in \mathcal{O}_a^b(\Cd)$, then for every $z \in \Cd$ it holds that
    \begin{equation}\label{eq:liouville}
        \begin{split}
            \left | F(z) \prod_{j=1}^d e^{\frac{a_j+b_j}{2}z_j^2} \right | & = |F(z)| \prod_{j=1}^d e^{\frac{a_j+b_j}{2}x_j^2} \prod_{j=1}^d e^{-\frac{a_j+b_j}{2}y_j^2} \\
        & \lesssim \prod_{j=1}^d e^{\frac{b_j-a_j}{2}|z_j|^2}.
        \end{split}
    \end{equation}
    If $a=b$, then the right-hand side of Equation \eqref{eq:liouville} is bounded, whence Liouville's theorem applies and shows the existence of a constant $\nu \in \C$ such that $F=\nu \varphi$, where $\varphi$ is defined as in the statement of the Proposition. Hence, $\mathcal{O}_a^b(\Cd) \subseteq \C \varphi$. The reverse inclusion is trivial. In order to prove the last claim, we write for $j \in \{ 1, \dots, d \}$ and for $z \in \Cd$
    $$
    z' \coloneqq (z_1, \dots, z_{j-1},z_{j+1}, \dots, z_d)^T \in \C^{d-1}. 
    $$
    If $d=1$, we skip the definition of $z'$. We then have
    $$
    \left |F(z_1,\dots, z_d) e^{\frac{a_j+b_j}{2}z_j^2} \right | \lesssim e^{-\frac{b_j-a_j}{2}|z_j|^2}H(z'),
    $$
    where $H$ is a function of $z' \in \C^{d-1}$. Suppose that $a_j>b_j$ holds for some $j \in \{ 1,\dots, d \}$. Then, it follows from Liouville's theorem that the map $z_j \mapsto F(z_1,\dots, z_d)$ vanishes identically. Since $z'$ was arbitrary, we have $F=0$.    
\end{proof}

\subsection{Uniqueness sets for $\mathcal{O}^b(\C^d)$}

Let $\mathcal{S} \subseteq \mathcal{O}(\Cd)$ be a linear space. A set $\Lambda \subseteq \Cd$ is said to be a uniqueness set for $\mathcal{S}$ if for every $F \in \mathcal{S}$ it holds that
$$
(F(\lambda) = 0 \ \ \forall \lambda \in \Lambda) \implies F=0.
$$
Clearly, every open set in $\Lambda \subseteq \Cd$ is a uniqueness set for $\mathcal{O}(\Cd)$. If $\Lambda$ is contained in $\Rd$, then the following can be said \cite[Lemma 3.4]{grohsLiehr4}.

\begin{lemma}\label{lma:elementary_uniqueness}
Let $\Lambda \subseteq \Rd$ be a Lebesgue-measurable set such that $\mathcal{L}^d(\Lambda) > 0$, where $\mathcal{L}^d$ denotes the $d$-dimensional Lebesgue measure. Then $\Lambda$ is a uniqueness set for $\mathcal{O}(\Cd)$.
\end{lemma}

Our objective now is to establish uniqueness sets for the class $\mathcal{O}^b(\C^d), \, b \in \R^d_{>0}$, which are discrete subsets of $\Rd$. To achieve this, we introduce the zero-counting function. Given $0 < r < R$ and an analytic function $F$ defined on the open unit ball $B_R(0) = { z \in \C : |z| < R }$, we denote by $n(r)$ the number of zeros $F$ within the region where $|z| \leq r$ (counting multiplicities). This defines a map $n : [0,R) \to \N_0$. Jensen's formula relates the distribution of zeros of an analytic function to its growth \cite[p. 50]{Young}.

\begin{theorem}[Jensen]
Let $0<r<R$ and suppose that $F \in \mathcal{O}(B_R(0))$ satisfies $F(0) \neq 0$. If $z_1,\dots, z_n$ are the zeros of $F$ for which $|z_j| \leq r$ (counting multiplicities), then
\begin{equation*}
    \begin{split}
        \frac{1}{2\pi} \int_0^{2\pi} \log |F(re^{i\theta})| \, d\theta & = \log |F(0)| + \sum_{j=1}^n \log \left ( \frac{r}{|z_j|} \right ) \\
        &= \log |F(0)| + \int_0^r \frac{n(t)}{t} \, dt.
    \end{split}
\end{equation*}
\end{theorem}

Utilizing Jensen's formula, we are well-equipped to prove the following uniqueness statement.

\begin{proposition}\label{lma:uniqueness_complex_variables}
Let $b = (b_1,\dots,b_d) \in \R^d_{>0}$. For each $j \in \{ 1,\dots, d \}$ let
$$
\Lambda_j \coloneqq \{ \pm \lambda_j(k) : k \in \N_0 \} \subseteq \R
$$
where $\lambda_j : \N_0 \to \R_{>0}$ is an increasing function. If
$$
\liminf_{k \to \infty} \frac{\lambda_j(k)}{\sqrt{k}} < \frac{1}{\sqrt{b_j e}}
$$
for every $j \in \{ 1,\dots, d \}$, then $\Lambda \coloneqq \Lambda_1 \times \dots \times \Lambda_d$ is a uniqueness set for $\mathcal{O}^b(\Cd)$. If on the other hand,
$$
\liminf_{k \to \infty} \frac{\lambda_j(k)}{\sqrt{k}} > \sqrt{\frac{\pi}{b_j}}
$$
for some $j \in \{ 1,\dots, d \}$, then $\Lambda \coloneqq \Lambda_1 \times \dots \times \Lambda_d$ is not a uniqueness set for $\mathcal{O}^b(\Cd)$.
\end{proposition}
\begin{proof}
\textbf{Step 1.} We prove the first part of the statement via induction over the dimension $d \in \N$. To this end, let $d=1$, let $\Lambda \coloneqq \Lambda_1$ with $\Lambda = \{ \pm \lambda(k) : k \in \N_0 \}$, and let $b \coloneqq b_1$. Suppose that $F : \C \to \C$ is an entire function which satisfies the bound  $|F(z)| \lesssim e^{b|z|^2}$. Without loss of generality we may suppose that $F(0)=1$. For if $F(0) \notin \{ 0,1 \}$, then one may use a scaling argument, and if $F$ has a zero of order $m$ at zero then one may consider $F(z)/z^m$. Let $n(r)$ denote the number of zeros of $F$ in the closed ball $\overline{B_r(0)} \coloneqq \{ z \in \C : |z| \leq r\} \subseteq \C$ of radius $r>0$. Combining the assumption that $F(0)=1$ with Jensen's formula implies the identity
$$
\int_0^r \frac{n(t)}{t} \, dt = \frac{1}{2\pi} \int_0^{2\pi} \log |F(re^{i\theta})| \, d\theta.
$$
Since $|F(z)| \leq C e^{b|z|^2}$ for some constant $C>0$, the previous equality yields the estimate
$$
\int_0^r \frac{n(t)}{t} \, dt \leq \log C + br^2.
$$
Now let $s>1$. The property of $t \mapsto n(t)$ being non-decreasing shows that
$$
n(r)\log s = n(r) \int_r^{sr} \frac{1}{t} \, dt \leq \int_r^{sr} \frac{n(t)}{t} \, dt \leq \log C + bs^2r^2.
$$
Assume that $F$ vanishes on $\Lambda \setminus \{ 0 \}$ (excluding zero because of the assumption that $F(0)$ is, without loss of generality, equal to one), but does not vanish identically. Choosing $r=\lambda(k)$ and $s>1$ yields $n(r) \geq  2k$ and
$$
2k \log s \leq \log C + bs^2\lambda(k)^2.
$$
This in turn shows the validity of the inequality
\begin{equation}\label{eq:square_root_bound}
    \sqrt{\frac{2\log s}{b s^2}} \leq \sqrt{\frac{\log C}{b s^2k}} + \frac{\lambda(k)}{\sqrt{k}},
\end{equation}
where we used the elementary estimate $\sqrt{p+q} \leq \sqrt{p}+\sqrt{q}$ for $p,q \geq 0$.
Maximizing the left-hand side of Equation \eqref{eq:square_root_bound} with respect to $s>1$ gives
$$
\sup_{s>1} \sqrt{\frac{2\log s}{bs^2}} = \frac{1}{\sqrt{be}},
$$
resulting in the lower $\liminf$-bound
$$
\frac{1}{\sqrt{be}} \leq \liminf_{k \to \infty} \frac{\lambda(k)}{\sqrt{k}}.
$$
Hence, if
$$
\liminf_{k \to \infty} \frac{\lambda(k)}{\sqrt{k}} < \frac{1}{\sqrt{be}},
$$
then this gives a contradiction and $F$ must vanish identically.

To perform the induction step, we suppose that the assertion holds for $d-1$. For a fixed $p \in \Lambda_1$, the map
$$
(z_2,\dots, z_d) \mapsto F(p,z_2,\dots,z_d)
$$
is an entire function of $d-1$ complex variables which belongs to $\mathcal{O}^{b'}(\C^{d-1})$, where $b' = (b_2,\dots, b_d)^T \in \R_{>0}^{d-1}$. This function vanishes on $\Lambda_2 \times \dots \times \Lambda_d$. By the induction hypothesis, it must vanish identically. Now fix $(p_2,\dots, p_d) \in \C^{d-1}$ and consider the map
$$
S(z) \coloneqq F(z,p_2,\dots,p_d).
$$
Then $S$ is an entire function which vanishes on $\Lambda_1$ and satisfies the bound
$$
|S(z)| \lesssim \left ( \prod_{j=2}^d e^{b_j|p_j|^2} \right ) e^{b_1 y_1^2}, \quad z=x_1 + i y_1.
$$
Therefore, it holds that $S \in \mathcal{O}^{b_1}(\C)$. The case $d=1$ shows that $S$ vanishes identically and the arbitrariness of $(p_2,\dots, p_d) \in \C^{d-1}$ concludes the proof of the statement.

\textbf{Step 2.} To prove the second part of the statement, we again start with $d=1$. Let $\Lambda \coloneqq \Lambda_1$ with $\Lambda = \{ \pm \lambda(k) : k \in \N_0 \}$, and let $b \coloneqq b_1$.
Suppose that $\liminf_{k \to \infty} \frac{\lambda(k)}{\sqrt{k}} > \sqrt{\frac{\pi}{b}}$. Let $\varepsilon>0$ such that $\varepsilon \in (0,b)$, and that
$$
\liminf_{k \to \infty} \frac{\lambda(k)}{\sqrt{k}} \geq \sqrt{\frac{\pi}{b-\varepsilon}}.
$$
It follows that there exists an integer $K \in \N$ such that
\begin{equation}\label{gegen}
    \lambda(k)^2 \geq \frac{\pi}{b-\varepsilon} k, \quad k \geq K.
\end{equation}
Define $\gamma(k) \coloneqq \lambda(k)^2$ and $\Gamma \coloneqq \{ \gamma(k) : k \geq K \}$. The Weierstrass factorization of the hyperbolic sine is given by
$$
\sinh(z) = z \prod_{k=1}^\infty \left ( 1 + \frac{z^2}{\pi^2 k^2} \right ).
$$
The lower bound given in \eqref{gegen} implies that
\begin{equation}
    \begin{split}
        \sinh(|z|) \gtrsim \prod_{k=K}^\infty \left ( 1 + \frac{|z|^2}{(b-\varepsilon)^2 \gamma(k)^2} \right ) \geq \left | \prod_{k=K}^\infty \left ( 1 - \frac{z^2}{(b-\varepsilon)^2 \gamma(k)^2} \right ) \right |.
    \end{split}
\end{equation}
Defining $F(z) \coloneqq \prod_{k=K}^\infty 1 - \frac{z^2}{\gamma(k)^2}$, it follows from the previous estimate and the definition of the hyperbolic sine, that
$$
|F(z)| \lesssim \sinh((b-\varepsilon)|z|) \lesssim e^{(b-\varepsilon)|z|}.
$$
Observe that $F$ is a non-trivial entire function that vanishes on $\Gamma$ (in fact, it vanishes on $\Gamma \cup (-\Gamma)$). Now define
$$
\tilde F(z) \coloneqq \prod_{k=1}^{K-1} \left ( 1 - \frac{z^2}{\lambda(k)^2} \right ) F(z^2).
$$
Then $\tilde F$ vanishes on $\{ \pm \sqrt{\gamma(k)} : k \geq K \} \cup \{ \pm \lambda(k) : k=1,\dots, K-1 \} = \Lambda$. Moreover,
$$
|\tilde F(z)| \lesssim e^{b|z|^2},
$$
which shows that $\tilde F \in \mathcal{O}^b(\C)$. Since $\tilde F$ is not the zero function, it follows that $\Lambda$ is not a uniqueness set for $\mathcal{O}^b(\C)$. The general statement for $\mathcal{O}^b(\Cd)$ follows from the one for $\mathcal{O}^b(\C)$: if $\Lambda_j$ satisfies $\liminf_{k \to \infty} \frac{\lambda_j(k)}{\sqrt{k}} > \sqrt{\frac{\pi}{b_j}}$, then there exists a non-trivial function $\tilde F \in \mathcal{O}^{b_j}(\C)$ that vanishes on $\Lambda_j$. Consequently, the function
$$
H(z_1, \dots, z_d) \coloneqq \tilde F(z_j)
$$
is non-trivial, it vanishes on $\C^{j-1} \times \Lambda_j \times \C^{d-j} \supseteq \Lambda_1 \times \cdots \times \Lambda_d$, and it belongs to $\mathcal{O}^{b}(\C^d)$. This yields the assertion.
\end{proof}

\section{Proof of the main results}\label{sec:3}

This section is devoted to the proofs of Theorem \ref{thm:main_result}, Corollary \ref{cor:polynomial_times_gauss}, and Corollary \ref{cor:hermite}. After proving the statements, we provide a list of remarks concerning extensions and comparisons to previous and future work.
We start with the proof of Theorem \ref{thm:main_result}. To this end, we introduce the following notation: if $p,q \in [1,\infty]$ are conjugate to each other, i.e., $\frac{1}{p}+\frac{1}{q}=1$, and if $v \in L^p(\Rd)$, $w \in L^q(\Rd)$, we define
$$
\langle v,w \rangle \coloneqq \langle v,w \rangle_{L^p(\Rd) \times L^q(\Rd)} \coloneqq \int_\Rd v(t)\overline{w(t)} \, dt.
$$
For a function $f : \Cd \to \C$, the shift of $f$ by $\nu \in \Cd$ is defined via
$$
T_\nu f(t) \coloneqq f(t-\nu),
$$
and its reflection is the map
$$
\mathcal{R}f(t)=f(-t).
$$
The Fourier transform of a function $f \in L^1(\Rd)$ is defined by
$$
\ft f(\omega) = \int_\Rd f(t)e^{-2 \pi i \omega \cdot t} \, dt
$$
where $\omega \cdot t = \sum_{j=1}^d \omega_j t_j$. The extension of $\ft$ from $L^1(\Rd) \cap \ltd$ to a unitary operator on $\ltd$ is carried out in the usual way. Notice that the short-time Fourier transform satisfies the identity $V_gf(x,\omega) = \ft(f \overline{T_x g} )(\omega)$. Finally, we define for a vector $\omega \in \Cd$ and $f : \Cd \to \C$ the map $f_\omega : \Cd \to \C$ by
$$
f_\omega(t) \coloneqq (T_\omega f(t))\overline{f( \overline{t} )}.
$$
We call $f_\omega$ the tensor product of $f$. If $f$ is assumed to be entire, then for every $\omega \in\C^d$, and every $t \in \C^d$, both maps
$$
\omega \mapsto f_\omega(t), \quad t \mapsto f_\omega(t),
$$
define entire functions from $\C^d$ to $\C$.

In order to establish the proof of Theorem \ref{thm:main_result}, a series of preliminary statements is needed. The first of these is a well-known result that relates STFTs of the form $V_ff$ to the equivalence relation $\sim$ \cite{Groechenig,auslander}.

\begin{lemma}\label{lemma:tensor_equality}
    Let $f,h \in \ltd$ such that $V_ff(x,\omega) = V_hh(x,\omega)$ for every $(x,\omega) \in \R^{2d}$. Then $f \sim h$. 
\end{lemma}

The second lemma pertains to parameter integrals with respect to a complex variable \cite{swi,mattner}.

\begin{lemma}\label{lma:holomorphic_integral}
 Let $F : \Rd \times \C^d \to \C$ be a function subjected to the following assumptions:
 \begin{enumerate}
     \item $F(\cdot,z)$ is Lebesgue-measurable for every $z \in \C^d$,
     \item $F(t,\cdot)$ is an entire function for every $t \in \Rd$,
     \item The function $z \mapsto \int_\Rd |F(t,z)| \, dt$ is locally bounded, that is, for every $z_0 \in \C^d$ there exists a positive constant $\delta > 0$ such that
     $$
     \sup_{\substack{z \in \C^d \\ |z-z_0| \leq \delta}} \int_\Rd |F(t,z)| \, dt < \infty.
     $$
 \end{enumerate}
Then $z \mapsto \int_\Rd F(t,z) \, dt$ defines an entire function of $d$ complex variables.
\end{lemma}

To proceed, we combine Lemma \ref{lemma:tensor_equality} with Lemma \ref{lma:holomorphic_integral} to make a first assertion on STFT phase retrieval with window in $\mathcal{O}_a^b(\Cd)$. This analysis is made under the assumption that complete spectrograms are available. The discretization step is performed in the subsequent stage.

\begin{corollary}\label{cor}
    Let $a,b \in \R_{>0}^d$, and let $0 \neq \varphi \in \mathcal{O}_a^b(\Cd)$. If $f,h \in \ltd$ are such that $|V_\varphi f(x,\omega)| = |V_\varphi h(x,\omega)|$ for every $(x,\omega) \in \R^{2d}$, then $f \sim h$.
\end{corollary}
\begin{proof}
    Consider the absolute value of the STFT of a function $f \in \ltd$ with respect to the window function $\varphi$ in the second argument, i.e., the map $\omega \mapsto |V_\varphi f(x,\omega)|^2$ for some fixed $x \in \Rd$. Using the definition of the STFT as the Fourier transform of the product of $f$ with a shift of $\varphi$ yields the identity
\begin{equation}\label{eq:spec_2nd_arg1}
    \begin{split}
        |V_\varphi f(x,\cdot)|^2 & = \ft(f\overline{T_x\varphi }) \overline{\ft(f\overline{T_x\varphi })} = \ift \mathcal{R} (f\overline{T_x\varphi }) \ift (\overline{f}T_x\varphi ) \\
        & = \ift ( \mathcal{R} (f\overline{T_x\varphi }) * (\overline{f}T_x\varphi )).
    \end{split}
\end{equation}
The convolution $\mathcal{R} (f\overline{T_x\varphi }) * (\overline{f}T_x\varphi )$ evaluated at $s \in \Rd$ is given by
\begin{equation}\label{eq:spec_2nd_arg2}
    \begin{split}
        \mathcal{R} (f\overline{T_x\varphi }) * (\overline{f}T_x\varphi )(s) & = \int_\Rd f(-(s-t)) \overline{\varphi (-(s-t)-x)} \overline{f(t)} \varphi (t-x) \, dt \\
        & = \int_\Rd f_s(t) \overline{T_x \varphi _s(t)} \, dt \\
        & = \langle f_s , T_x \varphi _s \rangle_{L^1(\Rd) \times L^\infty(\Rd)} \\
        & = \langle f_s , T_x \varphi _s \rangle.
    \end{split}
\end{equation}
Combining Equation \eqref{eq:spec_2nd_arg1} with Equation \eqref{eq:spec_2nd_arg2} shows that the spectrogram satisfies the identity
\begin{equation}\label{eq:ift}
    |V_\varphi f(x,\omega)|^2 = \ift(t \mapsto \langle f_t, T_x \varphi_t \rangle)(\omega).
\end{equation}
Hence, if $h \in \ltd$ is such that $|V_\varphi f| = |V_\varphi h|$, then
$$
\langle f_\omega, T_x \varphi_\omega \rangle = \langle h_\omega, T_x \varphi_\omega \rangle, \quad (x,\omega) \in \R^{2d}.
$$
It follows, that for every fixed $\omega \in \Rd$, the map
$$
Q_\omega(x) \coloneqq \langle f_\omega - h_\omega, T_x \varphi_\omega \rangle = (f_\omega - h_\omega) * (\mathcal{R}(\varphi_\omega))(x)
$$
vanishes identically. Applying the Fourier transform to $Q_\omega$ yields the relation
$$
\ft Q_\omega(x) = \ft(f_\omega - h_\omega)(x)\ft (\mathcal{R}(\varphi_\omega))(x) = 0,
$$
which holds for every $x \in \Rd$. Observe that the function $\mathcal{R}(\varphi_\omega)$ has Gaussian decay and since $\varphi \neq 0$ it does not vanish identically. Applying Lemma \ref{lma:holomorphic_integral}, shows that the map $x \mapsto \ft (\mathcal{R}(\varphi_\omega))(x)$ extends from $\Rd$ to an entire function which does not vanish identically. In particular, the set
$$
Z = \{ x \in \Rd : \ft (\mathcal{R}(\varphi_\omega))(x) = 0 \}
$$
has $d$-dimensional Lebesgue measure zero (cf. Lemma \ref{lma:elementary_uniqueness}). Since $\ft(f_\omega - h_\omega)$ is continuous, it follows that $\ft(f_\omega - h_\omega)$ must vanish identically. The arbitrariness of $\omega \in \Rd$ shows that $\ft(f_\omega - h_\omega)$ vanishes identically for every $\omega \in \Rd$. In other words,
$$
V_ff(\omega,y) = V_hh(\omega,y)
$$
for every $(\omega,y)\in \R^{2d}$. The statement follows from Lemma \ref{lemma:tensor_equality}.

\end{proof}

The next lemma constitutes a key component in the proof of Theorem \ref{thm:main_result}. It asserts, that the modulus squared of the short-time Fourier transform, $|V_\varphi f|^2$, extends from $\R^{2d}$ to an entire function belonging to $\mathcal{O}^c(\C^{2d})$, provided that $\varphi \in \mathcal{O}_a^b(\Cd)$. The constant $c \in \R_{>0}^{2d}$ is expressed in terms of $a$ and $b$.

\begin{lemma}\label{lma:stft_inclusion}
    Let $a,b \in \R_{>0}^d$, let $\varphi \in \mathcal{O}_a^b(\Cd)$, and let
    \begin{equation}\label{eq:c}
        c \coloneqq \left ( 2b_1 , \dots, 2b_d, \frac{2\pi^2}{a_1}, \dots, \frac{2\pi^2}{a_d}  \right  ) \in \R_{>0}^{2d}.
    \end{equation}
    Then for every $f \in \ltd$, it holds that $|V_\varphi f|^2 \in \mathcal{O}^c(\C^{2d})$, i.e, $|V_\varphi f|^2$ extends from $\R^{2d}$ to an entire functions belonging to the space $\mathcal{O}^c(\C^{2d})$.
\end{lemma}
\begin{proof}
\textbf{Step 1.} Let $z = x+i\rho \in \C^d$ with $x,\rho \in \R^d$ and consider the modulus of a complex shift of $\varphi_t$ by $z$, i.e., the map $|T_{z} \varphi_t|$. Employing the definition of the class $\mathcal{O}_a^b(\C^d)$ yields the estimate
\begin{equation*}
    \begin{split}
        |T_{z} \varphi_t(y)| & = |\varphi(y-(x+i\rho)-t) | |\varphi(\overline{y-(x+i\rho)})| \\
        & \lesssim \prod_{j=1}^d e^{-a_j(y_j-x_j-t_j)^2} \prod_{j=1}^d e^{b_j\rho_j^2} \prod_{j=1}^d e^{-a_j(y_j-x_j)^2} \prod_{j=1}^d  e^{b_j\rho_j^2} \\
        & = \prod_{j=1}^d e^{-2a_j(y_j-x_j-\frac{1}{2}t_j)^2}e^{-\frac{a_j}{2}t_j^2}e^{2b_j\rho_j^2}.
    \end{split}
\end{equation*}
The modulus of the inner product of an arbitrary function $u \in L^1(\Rd)$ with a complex shift of $\varphi_t$ by $z = x+i\rho$ is therefore upper bounded by
\begin{equation}\label{eq:3}
    \begin{split}
        |\langle u, T_{z} \varphi_t \rangle | & \lesssim \int_\Rd |u(y)| \prod_{j=1}^d e^{-2a_j(y_j-x_j-\frac{1}{2}t_j)^2}e^{-\frac{a_j}{2}t_j^2}e^{2b_j\rho_j^2} \, dy \\
    & \leq \prod_{j=1}^d e^{-\frac{a_j}{2}t_j^2} e^{2b_j \rho_j^2} \| u \|_{L^1(\Rd)}.
    \end{split}
\end{equation}

Now replace in the previous estimate the function $u$ by the tensor product $f_t$ where $f \in \ltd$. The Cauchy-Schwarz inequality shows that
\begin{equation}\label{eq:replacement}
    \| f_t \|_{L^1(\Rd)} \leq \| f \|_{L^2(\Rd)}^2.
\end{equation}
Consider the Fourier integral
\begin{equation}\label{eq:fourier_integral}
    F(z,z') \coloneqq \int_\Rd \langle f_t, T_z \varphi_t \rangle e^{2 \pi i z' \cdot t} \, dt.
\end{equation}
For $z=v+i\delta \in \C^d$ with $v,\delta \in \R^d$, the Fourier integral $F$ satisfies the estimate
\begin{align*}
        |F(z,z')| & \textoverset[0]{}{\leq} \int_\Rd |\langle f_t, T_z \varphi_t \rangle| \prod_{j=1}^d e^{-2\pi \delta_j t_j} \, dt \\
        & \textoverset[0]{Eq. \eqref{eq:3},\eqref{eq:replacement}}{\leq} \| f \|_{L^2(\Rd)}^2 \prod_{j=1}^d e^{2b_j\rho_j^2} \int_\Rd \prod_{j=1}^d e^{-\frac{a_j}{2} t_j^2} e^{-2 \pi \delta_j t_j} \, dt \\
        & \textoverset[0]{}{=} \prod_{j=1}^d e^{2b_j\rho_j^2} \prod_{j=1}^d e^{\frac{2\pi^2}{a_j}\delta_j^2} \| f \|_{L^2(\Rd)}^2 \int_\Rd \prod_{j=1}^d e^{-\frac{a_j}{2}(t_j+\frac{2\pi\delta_j}{a_j})^2} \, dt \\
        & \textoverset[0]{}{=} \prod_{j=1}^d e^{2b_j\rho_j^2} \prod_{j=1}^d e^{\frac{2\pi^2}{a_j}\delta_j^2} \| f \|_{L^2(\Rd)}^2 \int_\Rd \prod_{j=1}^d e^{-\frac{a_j}{2}t_j^2} \, dt \\
        & \textoverset[0]{}{\leq} C(a) \| f \|_{L^2(\Rd)}^2 \prod_{j=1}^d e^{2b_j\rho_j^2} \prod_{j=1}^d e^{\frac{2\pi^2}{a_j}\delta_j^2}
\end{align*}
for some constant $C(a)$ depending only on $a$. Moreover, an application of Lemma \ref{lma:holomorphic_integral} shows that $F$ defines an entire function on $\C^{2d}$.

\textbf{Step 2.} We apply the bounds derived in Step 1 to the modulus squared of the STFT. To do so, we make use of Equation \eqref{eq:ift}, derived in the proof of Corollary \ref{eq:ift}. In terms of the function $F$ as defined as in Equation \eqref{eq:fourier_integral}, we have that
\begin{equation}
    |V_\varphi f(x,\omega)|^2 = \ift(t \mapsto \langle f_t, T_x \varphi_t \rangle)(\omega) = \int_\Rd \langle f_t, T_x \varphi_t \rangle e^{2 \pi i \omega \cdot  t} \, dt = F(x,\omega),
\end{equation}
for every $(x,\omega) \in \R^{2d}$. According to Step 1, the function $F$ extends from $\R^{2d}$ to an entire function of $2d$ complex variables and satisfies the growth estimate
$$
|F(z,z')| \lesssim \prod_{j=1}^d e^{2b_j|z_j|^2} \prod_{j=1}^d e^{\frac{2\pi^2}{a_j}|z_j'|^2}.
$$
This implies that for $c \in \R_{>0}^{2d}$ given as in Equation \eqref{eq:c}, we have $|V_\varphi f|^2 \in \mathcal{O}^c(\C^{2d})$.
\end{proof}

We are ready to proof Theorem \ref{thm:main_result} from Section \ref{sec:main_results}.

\begin{proof}[Proof of Theorem \ref{thm:main_result}]
    The fact that (2) implies (1) is trivial. It remains to show that (1) implies (2). To this end, define $\Psi  \coloneqq M(\sqrt{\Z})^d$ and $\Gamma  \coloneqq N(\sqrt{\Z})^d$ with $M = \mathrm{diag}(\tau_1,\dots, \tau_d)$ and $N = \mathrm{diag}(\nu_1,\dots, \nu_d)$, so that $\Lambda = \Psi  \times \Gamma $. The choice of $\tau$ and $\nu$, in conjunction with Proposition \ref{lma:uniqueness_complex_variables} implies that $\Lambda$ is a uniqueness set for the space $\mathcal{O}^c(\C^{2d})$ with
\begin{equation}
        c \coloneqq \left ( 2b_1 , \dots, 2b_d, \frac{2\pi^2}{a_1}, \dots, \frac{2\pi^2}{a_d}  \right  ) \in \R_{>0}^{2d}.
    \end{equation}
    Let $f,h \in \lt$, such that
    \begin{equation}\label{eq:iddd}
        |V_\varphi f(\lambda)| = |V_\varphi h(\lambda)|, \quad \lambda \in \Lambda,
    \end{equation}
    or, equivalently,
    \begin{equation}
        |V_\varphi f(\lambda)|^2 = |V_\varphi h(\lambda)|^2, \quad \lambda \in \Lambda.
    \end{equation}
    According to Lemma \ref{lma:stft_inclusion}, it holds that $|V_\varphi f|^2 \in \mathcal{O}^c(\C^{2d})$ and $|V_\varphi h|^2 \in \mathcal{O}^c(\C^{2d})$. Since $\Lambda$ is a uniqueness set for $\mathcal{O}^c(\C^{2d})$, it follows that $|V_\varphi f|^2 = |V_\varphi h|^2$, i.e., the spectrograms of $f$ and $h$ with respect to the window function $\varphi$, agree everywhere on $\Rd$. Corollary \ref{cor} implies that $f \sim h$.
\end{proof}

With the proof of Theorem \ref{thm:main_result} now established, we can seamlessly proceed to derive the proof for Corollary \ref{cor:polynomial_times_gauss}.

\begin{proof}[Proof of Corollary \ref{cor:polynomial_times_gauss}]
    The map $x \mapsto e^{-\gamma \| x \|_2^2}$ is the restriction to $\Rd$ of the entire function $F(z) = \prod_{j=1}^d e^{-\gamma z_j^2}$. Defining
    $$
    \tilde \gamma \coloneqq  (\underbrace{\gamma, \dots, \gamma}_{d \ \mathrm{times}} )^T \in \R^d_{>0}
    $$
    it follows that $F \in \mathcal{O}_{\tilde \gamma}^{\tilde \gamma}(\Cd)$ (cf. Proposition \ref{prop:O_properies}(4)). Since $p$ is of exponential type, Lemma \ref{lemma:expType} shows that $pF \in \mathcal{O}_{\tilde \gamma-\tilde \varepsilon}^{\tilde \gamma + \tilde \varepsilon}(\Cd)$ whenever $\tilde \varepsilon$ is of the form
    $$
    \tilde \varepsilon \coloneqq  (\underbrace{\varepsilon, \dots, \varepsilon}_{d \ \mathrm{times}} )^T \in \R^d_{>0}
    $$
    for some $\varepsilon \in (0,\gamma)$. For a sufficiently small choice of $\varepsilon$, the assumptions on $\alpha$ and $\beta$ imply that
    $$
    \alpha < \sqrt{\frac{1}{2(\gamma+\varepsilon) e}}, \quad \beta < \sqrt{\frac{\gamma-\varepsilon}{2 \pi^2 e}}.
    $$
    The statement therefore follows from Theorem \ref{thm:main_result}.
\end{proof}

Notice that Corollary \ref{cor:hermite} follows at once from Corollary \ref{cor:polynomial_times_gauss} simply by the choice $\gamma=\pi$ and the fact that every polynomial is an entire function of exponential type.

\section{Beyond diagonal generating matrices}\label{sec4}

We aim to extend the results of Section \ref{sec:3} to the situation where the generating matrix of the square-root lattice is not necessarily diagonal. To this end, we recall that an invertible matrix $S \in \R^{2d \times 2d}$ is called symplectic if $S^T \mathcal{J}S = \mathcal{J}$ where $\mathcal{J}$ denotes the standard symplectic matrix
$$
\mathcal{J} \coloneqq \begin{pmatrix} 0 & -I_d \\ I_d & 0 \end{pmatrix},
$$
and $I_d$ denotes the identity matrix in $\R^{d \times d}$. The collection of all symplectic matrices in $\R^{2d \times 2d}$ is called the symplectic group and is denoted by $\mathrm{Sp}(2d,\R)$. If $d=1$, then $\mathrm{Sp}(2d,\R) = \mathrm{Sp}(2,\R) =  \mathrm{SL}(2,\R)$, the special linear group in $\R^{2 \times 2}$. The metaplectic group $\mathrm{Mp}(d)$ is the unitary representation of the double cover of the symplectic group $\mathrm{Sp}(2d,\R)$ on $\ltd$. Briefly speaking, to every symplectic matrix $S \in \mathrm{Sp}(2d,\R)$ corresponds a unitary operator $\hat S : \ltd \to \ltd$ (defined up to a sign factor), that satisfies the relation
$$
|V_gf(S(x,\omega))| = |V_{\hat S g} (\hat S f)(x,\omega)|,
$$
for every $f, g \in \ltd, \, S \in \mathrm{Sp}(2d,\R)$ and $(x,\omega) \in \R^{2d}$ \cite[Lemma 9.4.3]{Groechenig}. The application of metaplectic operators constitutes a classical technique in time-frequency analysis for transferring statements about separable lattices to non-separable ones. An extensive exposition on metaplectic operators can be found in a book by Folland \cite{Folland+2016}.

To make matters more concrete, we focus in the following on the STFT phase retrieval problem with Gaussian window and set the dimension to $d=1$. This regime has been studied the most extensively among STFT phase retrieval problems with different window functions. A general uniqueness result for sampling on square-root lattices reads as follows.

\begin{theorem}\label{gen}
    Let $S \in \mathrm{SL}_2(\R)$ and $p,q > 0$ be given by
    $$
    S = \begin{pmatrix}
        a & b \\
        c & d
    \end{pmatrix}, \quad p=(a^2+b^2)^{-1}, \quad q = |ac+bd|.
    $$
    If $p-q>0$ and if
    \begin{equation}\label{bound_cond}
        0 < \alpha <  \frac{1}{\sqrt{2\pi e}} \min \left \{ \frac{1}{\sqrt{p-q}}, \sqrt{p+q} \right \},
    \end{equation}
    then the following statements are equivalent for every $f,h \in \lt$:
    \begin{enumerate}
    \item $|V_\varphi f(\lambda)| = |V_\varphi h(\lambda)|$ for every $\lambda \in \alpha S (\sqrt{\Z})^{2}$,
    \item $f \sim h$.
\end{enumerate}
\end{theorem}
\begin{proof}
    Since $S \in \mathrm{SL}_2(\R)$, it follows that the condition
    $$
    |V_\varphi f(\lambda)| = |V_\varphi h(\lambda)|, \quad \lambda \in \alpha S (\sqrt{\Z})^{2}
    $$
    is equivalent to
    \begin{equation}\label{ididid}
        |V_{\hat S \varphi} (\hat S f)( \gamma )| = |V_{\hat S \varphi} (\hat S h)( \gamma )|, \quad \gamma \in \alpha (\sqrt{\Z})^{2}.
    \end{equation}
    According to \cite[Proposition 252]{gosson}, there exists a constant $C \in \C \setminus \{ 0 \}$ such that
    $
    \hat S \varphi(t) = Ce^{-\pi c t^2},
    $
    where $$
    c=(a^2+b^2)^{-1}(1+i(ac+bd)) = p(1+i(ac+bd)).
    $$
    Thus, for $x,y \in \R$ we have
    $$
    |\hat S \varphi(x+iy)| \lesssim e^{-\pi p (x^2 - y^2 - 2xy(ac+bd))} \leq e^{-\pi p x^2}e^{\pi p y^2} e^{2\pi x y q}.
    $$
    Using Young's inequality we obtain the bound
    $$
    |\hat S \varphi(x+iy)| \lesssim e^{-\pi (p-q) x^2}e^{\pi (p+q) y^2}.
    $$
    This shows that $\hat S \varphi \in \mathcal{O}^{\pi(p+q)}_{\pi(p-q)}(\C)$. Combining Theorem \ref{thm:main_result} with condition \eqref{bound_cond} and identity \eqref{ididid}, implies that $\hat S f \sim \hat S h$. Since the equivalence relation "$\sim$" is invariant under applications of invertible linear operators, it follows that $f \sim h$.
\end{proof}

The latter theorem provides a uniqueness statement for the STFT phase retrieval problem via sampling on lattices that are deformed under the action of a matrix in $\mathrm{SL}(2,\R)$. For instance, it covers the important class of rotation matrices and shear matrices.

\begin{example}[Rotated square-root lattices]
    Consider the rotation matrix $R_\theta \in \R^{2 \times 2}, \, \theta \in \R$, given by
\begin{equation}\label{eq:rotationmatrix}
    R_\theta \coloneqq \left( {\begin{array}{cc}
   \cos \theta & -\sin \theta \\
   \sin\theta & \cos\theta \\
  \end{array} } \right).
\end{equation}
Then $R_\theta \in \mathrm{SL}(2,\R)$ and the corresponding metaplectic operator $\widehat{R_\theta}$ is the so-called fractional Fourier transform of order $\theta$ \cite{gossonLuef,Namias}. For $S=R_\theta$, the constants $p$ and $q$ in Theorem \ref{gen} satisfy $p=1$ and $q=0$. Thus, for every $\theta \in \R$, every $0 < \alpha < \sqrt{\tfrac{1}{2\pi e}}$, and every $f,h \in \lt$ the following statements are equivalent:
\begin{enumerate}
    \item $|V_\varphi f(\lambda)| = |V_\varphi h(\lambda)|$ for every $\lambda \in \alpha R_\theta (\sqrt{\Z})^{2}$,
    \item $f \sim h$.
\end{enumerate}
We notice, that the previous equivalence is valid if $\varphi$ gets replaced by an arbitrary Hermite basis function. This follows from the fact, that Hermite functions are eigenfunctions for the Fractional Fourier transform corresponding to an eigenvalue $\nu \in \T$ \cite{Jaming,Namias}. 
\end{example}

\begin{example}[Sheared square-root lattices]
    Let $\sigma \in \R$, and let $A_\sigma \in \mathrm{SL}_2(\R)$ be given by
    $$
    A_\sigma \coloneqq \begin{pmatrix}
        1 & \sigma \\
        0 & 1
    \end{pmatrix}.
    $$
    Then $A_\sigma$ is a shear parallel to the $x$-axis. In this case, the values $p$ and $q$ of Theorem \ref{gen} are given by
    $$
    p = \frac{1}{1+\sigma^2}, \quad q=|\sigma|.
    $$
    It holds that $p-q > 0$ if and only if $\sigma \in (-r,r)$ where $r \approx 0.68$. For such $\sigma$, Theorem \ref{gen} shows uniqueness of the STFT phase retrieval problem with Gaussian window $\varphi(t)=e^{-\pi t^2}$ via sampling on a square root lattice generated by the shear matrix $\alpha A_\sigma$ where $\alpha>0$ satisfies \eqref{bound_cond}. Such a square-root lattice is depicted in Figure \ref{fig:2} for $\sigma=\frac{1}{2}$.
\end{example}

\begin{figure}
 \centering
 \hspace*{-1cm}
   \includegraphics[width=14cm,trim={0 1.5cm 0 0},clip]{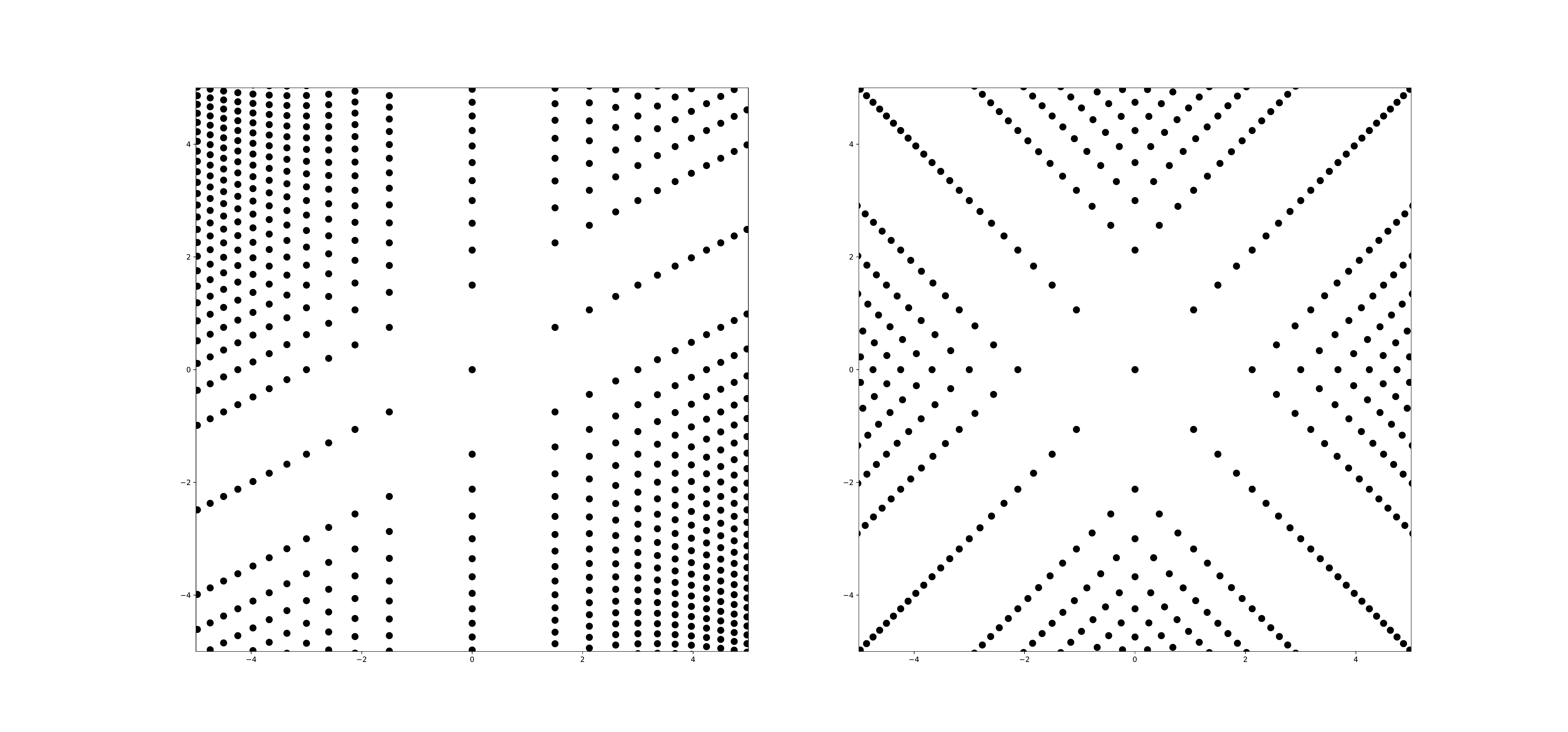}
 \caption{The left figure depicts a sheared square-root lattice. The right figure depicts a rotated square-root lattice.}
 \label{fig:2}
 \end{figure}

We end the paper with several remarks regarding comparisons to related work and applications.

\begin{remark}[Analyticity of the short-time Fourier transform]

Let $g$ be a window function and let $V_gf(x,\omega)$ be the STFT of $f$ with respect to $g$, evaluated at $(x,\omega) \in \R^{2d}$. For simplicity we assume in this remark that $d=1$. If we write $z=x+i\omega \in \C$, then the STFT gives rise to the map
\begin{equation}\label{eq:compl_id}
    V_gf : \C \to \C, \quad z=x+i\omega \mapsto V_gf(x,\omega).
\end{equation}
If $g$ is a Gaussian window function, then, modulo a multiplication by a non-zero weighting factor, Equation \eqref{eq:compl_id} defines an antiholomorphic function. Therefore, this case directly allows for the application of techniques from complex analysis and constitutes the main reason for the choice of Gaussian windows in most articles on STFT phase retrieval. A classical result of Asenci and Bruna shows that the Gaussian is essentially the only function which renders $V_gf$ into an analytic function in the above sense \cite[Theorem 2.1]{4839036}. Observe that the results obtained in the present article hold for a variety of window functions which are different from Gaussians (in fact, $\mathcal{O}_a^b(\Cd)$ is dense in $\ltd$ for $a<b$, as stated in Proposition \ref{prop:O_properies}). Our techniques therefore do not rely on the intimate relation to analytic functions induced by Gaussian windows.
\end{remark}

\begin{remark}[Support constraints]
Corollary \ref{cor:polynomial_times_gauss} states that if the window function is the product of a Gaussian with an entire function of exponential type, then every square-integrable function $f \in \ltd$ is determined up to a global phase by its phaseless STFT samples located on a separable square-root lattice. For precisely this choice of a window function, it is shown in \cite{grohsLiehrShafkulovska}, that ordinary separable lattices ($\Lambda = \alpha \Z^d \times \beta \Z^d$) yield uniqueness in classes of \emph{compactly supported} functions. Specifically, if $K \subseteq \Rd$ is a compact set and $\varphi$ is chosen as in Corollary \ref{cor:polynomial_times_gauss}, then there exist $\alpha,\beta>0$ such that the following two statements are equivalent for every $f,h \in L^2(K)$ (we stress the support constraint on $f$ and $h$):
\begin{enumerate}
    \item $|V_\varphi f(\lambda)| = |V_\varphi h(\lambda)|$ for every $\lambda \in \alpha \Z^d \times \beta \Z^d$,
    \item $f \sim h$.
\end{enumerate}
Because of the discretization constraints derived in \cite{grohsLiehrJFAA}, achieving an equivalence of the same nature becomes unattainable when substituting $L^2(K)$ with $\ltd$ and employing phaseless sampling on regular lattices. However, the findings in the current paper come into play in this scenario and demonstrate that uniqueness can be achieved by substituting the regular lattice with a square-root lattice.
\end{remark}

\begin{remark}[Applications and algorithm design]
Phase retrieval is a fundamental problem that arises in various fields, including coherent diffraction imaging, audio processing, and quantum mechanics. In these applications, only discrete samples are available, due to the limitations of digital devices, giving rise to the uniqueness problem investigated in the present article. While phase retrieval from the absolute value of the STFT on the entire time-frequency plane $\R^{2d}$ is well-documented, the uniqueness question becomes a substantial hurdle when only discrete samples are accessible. The outcomes of the present paper reveal the first uniqueness results for this problem in situations where no constraints on the underlying function space are imposed, shedding light on a previously unexplored aspect of phase retrieval. The square-root lattice scheme is designed to address the challenges posed by the uniqueness problem from discrete samples. The sampling points of this scheme exhibit the behavior of becoming more clustered the further one moves away from the origin in the time-frequency plane. This observation not only implies the potential benefits of denser sampling away from the origin in practical applications, but also underscores the importance of incorporating a notion of irregularity in sampling positions when designing new algorithms for phase retrieval. Unlike previous approaches that predominantly focused on lattice-based sampling, this insight suggests that embracing sampling irregularities can lead to more robust and effective phase retrieval algorithms tailored to real-world scenarios. Our results therefore provide a promising avenue for further research and practical applications in the aforementioned fields.
\end{remark}

\bibliographystyle{abbrv}
\bibliography{bibfile}

\end{document}